\documentclass{crm-eca}

\usepackage{graphicx}

\usepackage{Geomvap}  

\newcommand{\bR}{{\mathbb R}}

\newcommand{\cC}{{\mathcal C}}
\newcommand{\cS}{{\mathcal S}}
\newcommand{\cU}{{\mathcal U}}
\newcommand{\ts}{\tilde{s}}

\newtheorem{theorem}{Theorem}
\newtheorem{lemma}[theorem]{Lemma}

\newtheorem{proposition}[theorem]{Proposition}

\theoremstyle{definition}
\newtheorem{definition}[theorem]{Definition}

\theoremstyle{remark}

\begin{document}

\papertitle{Developable surfaces with prescribed boundary}


\paperauthor{Maria Alberich-Carrami\~nana}
\paperaddress{Departament de Matem\`atiques, Universitat Polit\`ecnica de Catalunya} 
\paperemail{maria.alberich{\\@}upc.edu}

\paperauthor{Jaume Amor\'os}
\paperaddress{Departament de Matem\`atiques, Universitat Polit\`ecnica de Catalunya} 
\paperemail{jaume.amoros{\\@}upc.edu}

\paperauthor{Franco Coltraro}
\paperaddress{Institut de Rob\`otica i Inform\`atica Industrial, CSIC-UPC} 
\paperemail{fcoltraro{\\@}iri.upc.edu}


\paperthanks{Research supported by project Clothilde, ERC research grant 741930, and research grants PID2019-103849GB-I00, from the Kingdom of Spain, 2017 SGR 932 from the Catalan Government.}

\makepapertitle

\Summary{It is proved that a generic simple, closed, piecewise regular curve in space can be the boundary of only finitely many developable surfaces with nonvanishing mean curvature. The relevance of this result in the context of the dynamics of developable surfaces is discussed.}


\section{Introduction}
The work presented here originated in the study by the authors of the motion and dynamics of pieces of cloth in the real world, with a view towards its robotic manipulation in a domestic, non-industrial, environment. 
Because in such an environment cloth is subject to low stresses, it makes sense to model garments as inextensible surfaces, and assume that their motion consists of isometries. The authors are currently developing such an isometric strain model, and its application to the control problem of cloth garments (\cite{Col20}).

The original state of cloth in such a model is flat, so the set of possible states of a piece of cloth in our model is the set of developable surfaces isometric to a fixed one, which we may assume to be a domain $R$ in the plane.

To study the dynamics of such a piece of cloth with the Lagrangian formalism, we need coordinates on the set of its states. This is usually done through a discretization scheme, but such schemes often introduce artifacts in the dynamics of cloth. Thus it would be interesting to have intrinsic, analytic coordinates on the space of states that are suitable for the formulation of a Lagrangian with a role analogous to that of the Helfrich Hamiltonian of membrane dynamics (see \cite{Des15}). As explained in the concluding section, such coordinates should allow the formulation of discretization schemes where the resulting mechanics are more independent of how the garment is meshed, more frugal in computation time, and closer to reality. 

Section \ref{s:coordenades} in this note discusses two candidates to the role of generalized coordinates in the space of states of a surface in our isometric strain model, and explains their common limitation from the viewpoint of their application.

Section \ref{s:vora} proposes an alternative approach: to track the motion of the surface by following its boundary. This is not straightforward because the boundary does not determine the position of the surface, but as we explain below our Main Theorem \ref{t:main} is a step in this direction.

\section{The space of developable surfaces} \label{s:coordenades}

{\em Developable surface} is a classical name for a smooth surface with Gaussian curvature 0. These are exactly the surfaces which are locally isometric to a domain in the Euclidean plane $\bR^2$. Developability places a strong constrain on a surface (see \cite{FP01}):
 
\begin{theorem}[structure theorem for developable surfaces, classical]\label{t:clasic}
A $\cC^3$ developable surface $S$ embedded in $\bR^3$ has an open subset which is ruled, with unit normal vector constant along each line of the ruling but varying in a transverse direction. Every connected component of its complement is contained in a plane.
\end{theorem}

This structure can be deduced from the Gauss map of the surface: Gaussian curvature 0 makes its rank 0 or 1, the latter rank being reached in an open subset of the surface. The normal vector is locally constant in the rank 0 subset. 

The dichotomy in the rank of the Gauss map, and varied classical notations, motivate

\begin{definition}
A developable surface is {\em torsal} if the Gauss map has rank 1 in a dense open subset.

{\em Flat patches} are connected subsets of a developable surface with nonempty interior where the Gauss map is constant, i.e. they are contained in a plane. 
\end{definition}

The subdivision of a developable surface into torsal and flat patches is given by the boundary of the vanishing locus of the mean curvature and is not necessarily simple.
 
\medskip

With a view to our intended applications, fix a planar domain $R \subset \bR^2$ which is compact, contractible and has a piecewise $\cC^{\infty}$ boundary. Usually $R$ will be a convex polygon. Define $\cS$ to be the set of all $\cC^3$ surfaces in $\bR^3$ isometric to $R$. These surfaces are all developable, and $\cS$ may be seen as the {\em space of states} of an inextensible (i.e. isometric for the inner distance) deformation of $R$ in Euclidean space.

The space of states $\cS$ can also be defined as the set of $\cC^2$ maps from $R$ to $\bR^3$ which are isometries with the image. As such, it is endowed with the compact-open topology derived from the Euclidean one in $R$ and $\bR^3$. This topology furnishes valuable tips for the study of $\cS$: the set of surfaces containing flat patches has an empty interior because there exist arbitrarily small deformations making the normal vector nonconstant on an open set. Torsal surfaces are {\em stably torsal} if the mean curvature function intersects transversely the zero function. These surfaces form an open subset $\cU \subset \cS$, and suffice for our practical study of $\cS$.

\medskip 

We can try to develop coordinates for the stably torsal state space $\cU$ based on the classical structure theorem. First, let us recall how to identify developable surfaces among the ruled ones

\begin{proposition}[classical, see \cite{Car76}]\label{p:lindep}
A ruled surface parametrized as $\phi(u,v)=\gamma(u)+v \cdot w(u)$, where $\gamma$ is a regular parametrized curve and $w$ a vector field over $\gamma$, is developable if and only if the 3 vectors $\gamma'(u),w(u),w'(u)$ are linearly dependent for all $u$.
\end{proposition}

Given a regular $\cC^2$ curve $\gamma$ there is a way to obtain systematically such rulings over $\gamma$ resulting in regular torsal surfaces:

\begin{proposition} \label{p:coordnormal}
Let $n$ be a unit normal $\cC^1$ vector field over a regular, $\cC^2$ curve with $n' \neq 0$. Then $w = n \times n'$ defines a torsal surface in a neighbourhood of $\gamma$. Moreover, all regular, torsal rulings over $\gamma$ are generated by such $w$, and only $n,-n$ define the same torsal surface.
\end{proposition}

\begin{proof}
$n$ is normal to $\gamma'$ and $w$ by their definitions, and $w'= n \times n''$ so at every $u$ the vectors $\gamma',w,w'$ are normal to $n$. Also, note that $w \neq 0$ because $n' \neq 0$.

If $\tilde{n}$ is another unit normal vector field such that $\tilde{n} \times \tilde{n}'= \mu w$ for some function $\mu(u)$ then note that $\tilde{n}$ has to be normal to both $w$ and $\gamma'$, hence a multiple of $n$.  

Finally, note that if $w$ is a nonvanishing tangent vector field over $\gamma$ defining a torsal surface around it, then we can select a unit vector field $n$ 
normal to $\gamma',w,w'$. The facts that $n$ is normal to $w$ and $w'$ imply that
$n'$ is also normal to $w$, so $n \times n'$ is a multiple of $w$.
\end{proof}

To define coordinates in the space of stably torsal surfaces $S$ isometric to a fixed bounded domain $R$, the pairs $(\gamma,n)$ of Prop. \ref{p:coordnormal} run into a practical difficulty: the condition that $n' \neq 0$ forces the Gauss map to have rank 1. If $S$ is a stably torsal surface with mean curvature $H$ of varying sign, we must subdivide it by the $H=0$ curves and parametrize separately each component of the complement. To follow a motion of the surface, one has to track the boundary shifts, mergers and splits of these components.

\medskip 

Ushakov proposes in \cite{Ush00} an alternative, PDE based, coordinate scheme viewing developable surfaces as solutions of the trivial Monge-Amp\`ere equation. But this requires parametrization of the surface in the form $z=z(x,y)$. Such parametrizations exist only locally, so their use leads even more intensely to the problem of tracking boundaries, mergers and splits of their subdomains.

\section{The boundary of a developable surface} \label{s:vora}

There is an alternative approach to study the dynamics of developable surfaces isometric to a fixed bounded planar domain $R$: follow the motion of the boundary $\partial R$ in space, and derive from this the developable surface that fills it. This leads to the

\medskip

\noindent {\bf Question.} Given a piecewise smooth simple closed curve $\gamma$ in $\bR^3$, what are the depelopable surfaces with boundary $\gamma$?

\medskip

The degeneracy nature of the trivial Monge-Amp\`ere equation makes it fail to have a unique solution for this kind of boundary problem. Indeed, it is easy to find examples where there is more than one solution, as shown in Fig. \ref{f:2reglatges}.

\begin{figure}
\includegraphics[scale=0.3]{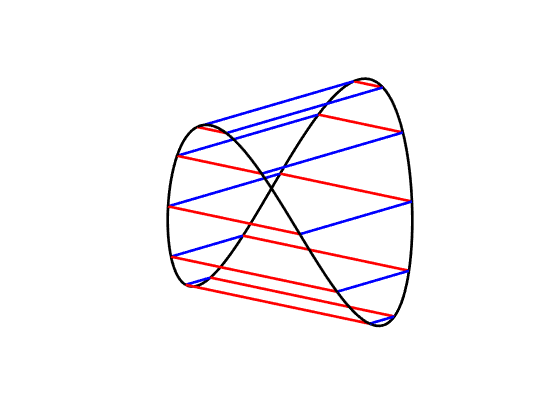}
\caption{A smooth simple closed curve (in black) which is the boundary of two developable 
surfaces (indicated in red and blue respectively)}
\label{f:2reglatges}
\end{figure}

Nevertheless, for problems such as the study of cloth dynamics it is not necessary that the boundary problem have a unique solution. It suffices to know that it will always have a finite set of solutions, because this solution set is then discrete, with different solutions separated by a nontrivial jump in any tagging energy, local coordinate \dots In such case, once one has a developable ruling with a boundary $\gamma_0$ at time $t=0$, the evolution $\gamma_t$ of the boundary will determine the analytic continuation of the $t=0$ developable ruling, and identify a unique ruling for every time $t$. Herein lies the interest of the authors in our

\begin{theorem}[Main Theorem]\label{t:main}
Let $\gamma$ be a simple closed curve in $\bR^3$ which is piecewise $\cC^2$, has nonvanishing curvature, torsion vanishing at finitely many points, and such that only for finitely many pairs $s \neq \ts$ does the tangent to $\gamma$ at $s$ pass through $\gamma(\ts)$. Then, there can be at most finitely many developable surfaces with boundary $\gamma$ and nonzero mean curvature in its interior. 
\end{theorem}

Let us point out that the preconditions that we impose on $\gamma$ are generic, i.e. satisfied by a dense open subset of the embeddings of $S^1$ in $\bR^3$.

The starting idea to prove the theorem is another classical result, analogous to Prop. \ref{p:lindep}:

\begin{lemma}\label{l:bitangent}
Let $S$ be a torsal surface with boundary $\gamma$, and $l \subset S$ a segment with endpoints $P,Q$ in $\gamma$. Then the common tangent plane to $S$ along $l$ is tangent to $\gamma$ at both $P,Q$.
\end{lemma}

The proof consists in pointing out that the normal vector to $S$ stays constant over the segment $l$, and that $\gamma$ is tangent to $S$.

Lemma \ref{l:bitangent} presents developable rulings as arcs of bitangent planes (i.e., tangent to $\gamma$ at 2 points). Such planes are given by pairs $s \neq \ts$ whose tangent lines are coplanar:

\begin{proposition} \label{p:morse}
Let $\gamma : [0,L] \subset \bR^3$ be a simple, closed, arc-parametrized $\cC^3$ curve. The function
\begin{align*}
D : [0,L]^2 &\longrightarrow \bR \\
(s,\ts) &\longmapsto \det \left( \gamma(s)-\gamma(\ts),\gamma'(s),\gamma'(\ts) \right)
\end{align*}
is a Morse function at a neighbourhood of its zeros $(s,\ts)$ such that: $s \neq \ts$, $\gamma$ has nonzero curvature and torsion at both $s,\ts$, and the tangent line to $\gamma$ in each one does not pass through the other point of the curve.
\end{proposition}

\begin{proof}
It is a straightforward computation. With coordinates $(s,\ts)$ we have that
\begin{equation*}
dD= \left( \det \left( \gamma(s)-\gamma(\ts),\gamma''(s),\gamma'(\ts) \right),
\det \left( \gamma(s)-\gamma(\ts),\gamma'(s),\gamma''(\ts) \right) \right)
\end{equation*}

Let $(s,\ts)$ be a zero of $D$ with $s \neq \ts$, which is also a critical point of $D$. If any of the linear subspaces spanned by $\gamma(s)-\gamma(\ts), \gamma'(\ts)$ and by $\gamma(s)-\gamma(\ts), \gamma'(s)$ has dimension less than 2, the tangent line to $\gamma$ at one of the points $\gamma(s),\gamma(\ts)$ contains the other.

When both linear subspaces have dimension 2, the conditions $D(s,\ts)=0, dD(s,\ts)=(0,0)$ show that $\gamma$ has the same osculating plane to $\gamma$ at the points $\gamma(s),\gamma(\ts)$. Because of this, the second differential of $D$ is
\begin{equation*}
d^2 D = \left( \begin{array}{cc}
\kappa_s \tau_s \det \left( \gamma(s)-\gamma(\ts),B_s,\gamma'(\ts) \right) & 0 \\
0 & \kappa_{\ts} \tau_{\ts} \det \left( \gamma(s)-\gamma(\ts),\gamma'(s), B_{\ts} \right)
\end{array} \right)
\end{equation*}
Here $\kappa,\tau,B$ are respectively the curvature, torsion, binormal vector of the Frenet frame, at the point given by their subindex. 
The determinants in the diagonal of $d^2D$ are nonzero because each consists of a binormal vector and a basis for the osculating plane at the same point of the curve.
\end{proof}

Proposition \ref{p:morse} has a version for piecewise $\cC^3$ $\gamma$, saying just that $D$ is Morse under the additional hypothesis that $s,\ts$ do not correspond to corner points, at which $D$ has two different definitions. We are now ready for

\begin{proof}[Proof of Theorem \ref{t:main}]
A torsal developable surface $S$ is foliated by segments which can only end at the boundary or at points of vanishing mean curvature. Having ruled out the latter, $S$ is determined by an arc of bitangent planes $B(t)$, with $t \in [a,b]$, such that the curves $s(B(t)), \ts(B(t))$ formed by the two points of tangency of $B(t)$ cover $\gamma$. 

Away from the finite set of horizontal and vertical lines in $[0,L]^2$ where one of the values $s,\ts$ corresponds to a corner point, point with vanishing torsion, or point whose tangent line intersects $\gamma$ again, the pairs of values $s \neq \ts$ for which there exists at all a bitangent plane to $\gamma$ through $\gamma(s),\gamma(\ts)$ lie by Proposition \ref{p:morse} in the zero set of a Morse function $D$ from an open subset of $[0,L]^2 \subset \bR^2$ to $\bR$. The function $D$ is proper, therefore it is Morse over a suitably small range of values $(-\varepsilon,\varepsilon)$, which implies that $D_0$ is a finite union of smooth curves with transverse intersections in $[0,L]^2$.

The arc $B(t)$ is determined by its tangency points curve $(s(B(t)), \ts(B(t))) \subset [0,L]^2$, which must lie in the union of $D_0$ and finitely many vertical and horizontal lines, and cover $\gamma$, i.e. $\gamma=s(B) \cup \ts(B)$. There are only finitely many possibilities for that, once we specify a beginning point for the curves $s(B),\ts(B)$.
\end{proof}

\section{Future continuation}

The authors hope to carry out the program outlined in this note: to subdivide a developable surface $S$ in patches according to the sign of its mean curvature, and follow its motion in a dynamical system by tracking the boundaries of the patches.

This approach is promising because it works with a 1-dimensional set of space coordinates which satisfy few restrictions, rather than with a 2-dimensional set of space coordinates that are heavily restricted because of the assumption of isometry. 

There is a second interest which is more analytical: can we identify the developable surfaces which minimize a potential such as the gravitatory potential?  Fixing the boundary and the area of the surface leads to the Poisson equation which is not as straightforward to solve as its 1-dimensional analogue (\cite{DH96}). What equation does one get if instead of fixing the area one fixes the Gaussian curvature to be zero? It is likely that global, i.e. parametrized by the fixed domain $R$, solutions will have in general singularities along points or curves.










\end{document}